\documentclass[leqno,12pt]{article} 
\usepackage{soul}
\usepackage{amsmath, amssymb}
\usepackage{amsthm} 
\usepackage{amssymb}
\usepackage{mathrsfs}
\usepackage{amssymb, url, color, graphicx, amscd, mathrsfs}
\usepackage[colorlinks=true, bookmarks=true, pdfstartview=FitH, pagebackref=true, linktocpage=true, linkcolor = magenta, citecolor = blue]{hyperref}
\usepackage{graphicx}
\usepackage{times}
\newcommand{\KN}{\mathbin{\bigcirc\mspace{-15mu}\wedge\mspace{3mu}}}
\theoremstyle{plain} 
\newtheorem{theorem}{\indent\sc Theorem}[section]
\newtheorem{lemma}[theorem]{\indent\sc Lemma}
\newtheorem{corollary}[theorem]{\indent\sc Corollary}
\newtheorem{proposition}[theorem]{\indent\sc Proposition}

\theoremstyle{definition} 
\newtheorem{definition}[theorem]{\indent\sc Definition}
\newtheorem{remark}[theorem]{\indent\sc Remark}
\newtheorem{example}[theorem]{\indent\sc Example}

\numberwithin{equation}{section}
\DeclareMathOperator{\Rm}{Rm}

\DeclareMathOperator{\Ric}{Ric}

\makeatletter
\def\address#1#2{\begingroup
\noindent\parbox[t]{7.8cm}{%
\small{\scshape\ignorespaces#1}\par\vskip1ex
\noindent\small{\itshape E-mail address}%
\/: #2\par\vskip4ex}\hfill%
\endgroup}%
\makeatother
\title{{Vanishing results from Lichnerowicz Laplacian on complete K\"{a}hler manifolds and applications}} 
\author{
\textsc{Gunhee Cho, Nguyen Thac Dung} 
}
\date{} 
\usepackage[pagewise]{lineno}
\begin{document}

\maketitle



\begin{abstract}
	In this paper, we show several rigidity results for harmonic $(p,q)$-forms in complete K\"{a}hler manifolds. We also give several applications to study non-compact K\"{a}hler manifolds with parallel Bochner tensor or quaternion K\"{a}hler manifolds. Our results are natural extensions of Petersen and Wink's results in \cite{PW21, PW} in the setting of complete, non-compact K\"{a}hler manifolds. 
\end{abstract}
\section{Introduction}

This paper is a next study conducted by the authors (see \cite{CDH}) in the K\"ahler manifold setting.

Denote $\Delta_{L}=\nabla^{*} \nabla+c \Ric$ as Lichnerowicz Laplacian for $c >0$. The precise description of $\Delta_{L}$ and all necessary definitions and notations are described in Section 2. As given in Section 2, let $\mathfrak{R}_{|\mathfrak{u}(m)}$ be K\"{a}hler curvature operator, then the first result of this paper is stated as follows.

	\begin{theorem}\label{thm1}
	Let $(M,g)$ be a complete, non-compact K\"ahler manifold. Assume that $g(\mathfrak{R}_{|\mathfrak{u}(m)}(T^\mathfrak{u}), \overline{T}^\mathfrak{u})$ is nonnegative for every harmonic $(0,k)$-tensor $T$. Then every harmonic tensor $T$ (with respect to the Lichnerowicz Laplacian) is parallel if $|T| \in L^Q(M)$ for some $Q\geq2$. 
\end{theorem}

With a presense of a negative lower bound of $g(\mathfrak{R}_{|\mathfrak{u}(m)}(T^\mathfrak{u}), \overline{T}^\mathfrak{u})$, we need some more assumption that the weighted Poincar\'e inequality holds (\cite[Def 0.1]{LPWJ}): for $M^n$ an $n$-dimensional complete Riemannian manifold, we say that $M$ satisfies a weighted Poincar\'e inequality with a nonnegative weight function $\rho$ on $M$, if the inequality
\begin{equation*}
\int_{M} \rho(x)\phi^2(x)dV \leq \int_{M} |\nabla \phi|^2 dV
\end{equation*}
is valid for any compactly supported smooth function $\phi \in C^{\infty}_0(M)$. 
We put two addtional hypotheses on $\rho$, 
\begin{equation}\label{eq:C1}
\liminf_{x\rightarrow \infty} \rho(x)>0,
\end{equation}
and 
\begin{equation}\label{eq:C2}
\text{$M$ is nonparabolic},
\end{equation}
i.e., there exists a symmetric positive Green's function $G(x,y)$ for the Laplacian acting on $L^2$ functions (otherwise, we say $M$ is parabolic). All assumptions on a weight function $\rho$ can be regarded as the generalization of the positivity condition of the first Dirichlet eigenvalue $\lambda_1(M)$ \cite{LPWJ}. {We refer the interested readers to \cite{LPWJ} for several examples regarding Riemannian manifolds with weighted Poincar\'{e} inequality and also see \cite{CXZD} to the of study minimal hypersurfaces. Some other examples are given by Minerbe in \cite{MV09}. Moreover, K\"{a}hler manifolds with weighted Poincar\'{e} inequality are also investigated in \cite{LW09, Mun07}. }

Our second result is formulated as follows.
\begin{theorem}\label{thm2}
	Let $(M,g)$ be a connected complete non-compact K\"ahler manifold. Assume that  $M$ satisfies a weighted Poincar\'e inequality with a nonnegative weight function $\rho$ with ~\eqref{eq:C1} and ~\eqref{eq:C2}, and also $g(\mathfrak{R}_{|\mathfrak{u}(m)}(T^\mathfrak{u}), \overline{T}^\mathfrak{u}) \geqslant -\kappa \rho|T|^2$ for all $(0,k)$-tensors $T$, where $\kappa\geq0$ is given. Then every harmonic tensor $T$ (with respect to the Lichnerowicz Laplacian) vanishes provided that $|T| \in L^Q(M), Q\geq2$ and $0\leqslant \kappa<\frac{4(Q-1)}{cQ^2}$. \end{theorem}

As a consequence of Theorem~\ref{thm2}, we have the following classification theorem which is motivated by the work of Bryant \cite{Bry} and Kamishima \cite{kam} in classifying compact and complete Bochner flat manifolds. 

\begin{theorem}\label{d2} 
	Suppose that $(M,g)$ is a complete K\"{a}hler {non-compact} manifold of complex dimension $n$ satisfying that the Bochner tensor is divergence free. If $M$ satisfies a weighted Poincar\'e inequality and
	$$\mu_1+\ldots+\mu_{\lfloor\frac{n+1}{2}\rfloor}+\frac{1+(-1)^n}{4}\mu_{\lfloor\frac{n+1}{2}\rfloor+1}\geq-k\rho$$
	then the Bochner tensor vanishes provided that its $L^Q$-norm is finite and $0\leq k<\frac{Q-1}{Q^2}, Q\geq2$. Consequently, if we assume furthermore that $M$ is the complete Bochner flat geometry as in \cite{kam} then $M$ must be one of the following forms:
		\begin{enumerate}
			\item Complex Euclidean geometry $(\mathbb{C}^n\rtimes {\rm U}(n), \mathbb{C}^n, g_\mathbb{C}, J_\mathbb{C})$,
			\item Product of complex hyperbolic and projective geometry $
			({\rm PU}(m, 1) \times {\rm PU}(n - m +1), \mathbb{H}^m_\mathbb{C}\times\mathbb{CP}^{n-m}, g_\mathbb{H}\times g_\mathbb{CP}, J_\mathbb{H}\times J_\mathbb{CP}), m = 0,\ldots, n$,
			\item Intransitive K\"{a}hler geometry $(\mathbb{C}^{n-k}\rtimes{\rm U}(n- k))\times {\rm U}(\ell_1,\ldots, \ell_m), \mathbb{C}^n, \hat{g}_a, J_\mathbb{C})$, $k\geq1$.
		\end{enumerate}
		Moreover, if the first two cases hold and $M$ is irreducible then $M$ is of constant holomorphic sectional curvature. Consequently, $M$ is $\mathbb{C}^n$, $\mathbb{CP}^n$, $\mathbb{B}^n$, or their quotients. Here we refer the reader to \cite{kam} for the notation of Bochner flat geometry and models.
\end{theorem}
Combining the vanishing results as in Theorem \ref{thm2} and a classification of K\"{a}hler-Ricci solitons with vanishing Bochner tensor \cite{yanhui2012}, we obtain the following proposition.
\begin{proposition}\label{d3}
	Suppose that $(M,g)$ is a complete gradient K\"{a}hler-Ricci soliton of complex dimension $n$ satisfying that the Bochner tensor is divergence free. If $M$ satisfies a weighted Poincar\'e inequality and
	$$\mu_1+\ldots+\mu_{\lfloor\frac{n+1}{2}\rfloor}+\frac{1+(-1)^n}{4}\mu_{\lfloor\frac{n+1}{2}\rfloor+1}\geq-k\rho$$
	then the Bochner tensor vanishes provided that its $L^Q$-norm is finite and $0\leq k<\frac{Q-1}{Q^2}, Q\geq2$. Consequently, $M$ is $\mathbb{C}^n$, $\mathbb{CP}^n$, $\mathbb{B}^n$, or their quotients.
\end{proposition}

The analogous {theorem} of the non-compact quaternion manifolds {is} also obtained. 

\begin{theorem}\label{thm2Q}
	Let $(M,g)$ be a connected complete non-compact quaternion K\"ahler manifold. Assume that  $M$ satisfies a weighted Poincar\'e inequality with a nonnegative weight function $\rho$ with ~\eqref{eq:C1} and ~\eqref{eq:C2}, and also $g(\mathfrak{R}_{|\mathfrak{sp}(m)\oplus \mathfrak{sp}(1)}(T^\mathfrak{sp}), \overline{T}^\mathfrak{sp}) \geqslant -\kappa \rho|T|^2$ for all $(0,k)$-tensors $T$, where $\kappa\geq0$ is given. Then every harmonic tensor $T$ (with respect to the Lichnerowicz Laplacian) vanishes provided that $|T| \in L^Q(M), Q\geq2$ and $0\leqslant \kappa<\frac{4(Q-1)}{cQ^2}$. \end{theorem}

{The paper is organized} as follows. In Section 2, we recall some basic important facts on Bochner techniques on K\"{a}hler manifolds which are inspired by the work of Petersen and Wink. In Section 3, we derive vanishing results and prove Theorems 1.1-1.2 and Theorems 1.5-1.6. Geometric applications are introduced in Section 4. In particular, we give a proof of Theorem \ref{d2} and Proposition \ref{d3}. 
\section{Preliminaries}\label{sec0}
\subsection{Tensors}
We collect the necessary ingredients from \cite[Section 1]{PW}.  
Let $(V,g)$ be an $n$-dimensional Euclidean vector space. The metric $g$ induces a metric on $\otimes^{k}V^{*}$ and $\wedge^k V$ in the way that if $\left\{e_i \right\}_{i=1,\cdots,n}$ is an orthonormal basis for $V$, then $\left\{e_{i_1}\wedge \cdots \wedge e_{i_k} \right\}_{{1\leq i_1< \cdots<i_k\leq n }}$ is an orthonormal basis for $\wedge^k V$. 

$\wedge^2 V$ inheries a Lie algebra structure from $\mathfrak{so}(V)$, and the induced Lie algebra action on $V$ is given by
$$(X \wedge Y)Z=g(X,Z)Y-g(Y,Z)X.$$ 
In particular, for $\Xi_\alpha,\Xi_\beta\in \wedge^2 V$, $$(\Xi_\alpha)\Xi_\beta=[\Xi_\alpha,\Xi_\beta].$$
	\begin{definition}\label{def22}
	Let $V_\mathbb{C}=V\bigotimes_\mathbb{R}\mathbb{C}$. For a complex valued, $\mathbb{R}$-multilinear tensor $T$ on $V$, namely, $T\in \bigotimes^rV^*_\mathbb{C}$ and $L\in\mathfrak{so}(V)$ set 
	$$LT(X_1, \ldots, X_r)=-\sum\limits_{i=1}^rT(X_1,\ldots, LX_i, \ldots, X_r).$$
	If $\mathfrak{g}\subset\mathfrak{so}(V)$ is a Lie subalgebra, we define $T^\mathfrak{g}\in\left(\bigotimes^rV^*_\mathbb{C}\right)\bigotimes_\mathbb{R}\mathfrak{g}$
	$$g(L, T^\mathfrak{g}(X_1,\ldots, X_r))=LT(X_1,\ldots, X_r)$$
	for all $L\in\mathfrak{g}\subset\mathfrak{so}(V)=\Lambda^2V$. 
\end{definition}

A tensor $\Rm \in \otimes^4 V^{*}$ is an algebraic curvature tensor if 
$$\Rm(X,Y,Z,W)=-\Rm(Y,X,Z,W)=-\Rm(X,Y,W,Z)=\Rm(Z,W,X,Y),$$
$$\Rm(X,Y,Z,W)+\Rm(Y,Z,X,W)+\Rm(Z,X,Y,W)=0.$$
In particular, it induces the curvature operator $\mathfrak{R} : \wedge^2 V \rightarrow \wedge^2 V$ via
$$g(\mathfrak{R}(X \wedge Y),Z\wedge W)=\Rm(X,Y,Z,W).$$
The associated symmetric bilinear form is denoted by $R\in {\rm Sym^2_B}(\wedge^2 V)$. We have 
$$|\Rm|^2=4|R|^2.$$

The curvature operator $\mathfrak{R}$ of a Riemannian manifold $(M,g)$ vanishes on the complement of the holonomy algebra $\mathfrak{hol}$. In particular, it induces $\mathfrak{R}_{|\mathfrak{hol}}: \mathfrak{hol}\to\mathfrak{hol}$ and the corresponding curvature tensor $R\in \rm Sym^2_B(\mathfrak{hol})$. If $\mathfrak{hol}=\mathfrak{u}(m)$, then $(M,g)$ is K\"ahler and the operator $\mathfrak{R}_{|\mathfrak{u}(m)}: \mathfrak{u}(m)\to\mathfrak{u}(m)$ is called K\"ahler curvature operator and the associated $R\in \rm Sym^2_B(\mathfrak{u}(m))$ is the K\"ahler curvature tensor. Every K\"ahler curvature tensor on a K\"ahler manifold $(M,g,J)$ with the complex sturcture $J$ satisfies
$$\Rm(X,Y,Z,W)=\Rm(JX,JY,Z,W)=\Rm(X,Y,JZ,JW).$$
If $\mathfrak{hol}=\mathfrak{sp}(m)\oplus \mathfrak{sp}(1)$, then $(M,g)$ is quaternion K\"ahler and the operator $\mathfrak{R}_{|\mathfrak{sp}(m)\oplus \mathfrak{sp}(1)}: \mathfrak{sp}(m)\oplus \mathfrak{sp}(1)\to\mathfrak{sp}(m)\oplus \mathfrak{sp}(1)$ is called quaternion K\"ahler curvature operator and the associated $R\in \rm Sym^2_B(\mathfrak{sp}(m)\oplus \mathfrak{sp}(1))$ is the quaternion K\"ahler curvature tensor.

Moreover, if $\mathfrak{R}: \mathfrak{g}\to\mathfrak{g}$ is a self-adjoint operator with orthonormal eigenbasis $\{\Xi_\alpha\}$ and corresponding eigenvalues $\{\lambda_\alpha\}$, then 
$$\mathfrak{R}(T^\mathfrak{g})=\mathfrak{R}\circ T^\mathfrak{g}=\sum\limits_\alpha\mathfrak{R}(\Xi_\alpha)\otimes\Xi_\alpha T,$$
as a consequence, we have
$$g(\mathfrak{R}(T^\mathfrak{g}), \overline{T}^\mathfrak{g})=\sum\limits_\alpha\lambda_\alpha|\Xi_\alpha T|^2$$
and in particular,
$$|T^\mathfrak{g}|^2=\sum\limits_\alpha|\Xi_\alpha T|^2.$$
In case $\mathfrak{g}=\mathfrak{u}(n)$ and $\mathfrak{g}=\mathfrak{sp}(m)\oplus \mathfrak{sp}(1)$, we will write $T^\mathfrak{u}$ and $T^\mathfrak{sp}$ to simplify notation respectively.

\subsection{Lichnerowicz Laplacian and the Bochner technique}
Let $(M,g)$ be a $n$-dimensional Riemannian manifold and denote $R(X,Y)Z=\nabla_{Y} \nabla_{X} Z-\nabla_{X} \nabla_{Y} Z+\nabla_{[X, Y]} Z$ as $(1,3)$- Riemmanian curvature tensor. We denote by $\mathcal{T}^{(0,k)}(M)$ the vector bundle of $(0, k)$-tensors on $M$. Recall that the \textit{Weitzenb\"{o}ck curvature operator} on a tensor $T \in \mathcal{T}^{(0,k)}(M)$ is defined by
$$
\Ric(T)\left(X_{1}, \ldots, X_{k}\right)=\sum_{i=1}^{k} \sum_{j=1}^{n}\left(R\left(X_{i}, e_{j}\right) T\right)\left(X_{1}, \ldots, e_{j}, \ldots, X_{k}\right).
$$

 For $c >0$ the Lichnerowicz Laplacion is given by 
 $$
 \Delta_{L}=\nabla^{*} \nabla+c \Ric.
 $$ 
 
 A tensor $T \in \mathcal{T}^{(0,k)}(M)$ is called harmonic if $\Delta_{L} T=0$.
 \begin{example}
 	There are some important example of Lichnerowicz Laplacian for different $c>0$. 
 	\begin{itemize}
 		\item[(a)] The Hodge Laplacian is a Lichnerowicz Laplacian for $c=1$. 
 		\item[(b)]\label{rmk1} For $c=\frac{1}{2}$ the Riemmannian curvature tensor $\Rm$ is harmonic if it is divergence free. The fact is that $\Rm$ is divergence free if and only if its Ricci tensor is a Codazzi tensor, in this case its scalar curvature is constant. If the manifold is Einstein, its Ricci tensor is always Codazzi. Therefore, the curvature tensors of Einstein manifolds are harmonic. 
 \end{itemize}
 \end{example}

\begin{proposition}\label{ric}
	Let $\mathfrak{R}:\Lambda^2TM\to\Lambda^2TM$ denote the curvature operator of $(M, g)$. If $\mathfrak{g}\subset\mathfrak{so}(m)$ denotes the holonomy algebra, then $\mathfrak{R}_{|\mathfrak{g}}: \mathfrak{g}\to \mathfrak{g}, \mathfrak{R}_{|\mathfrak{g}^\perp}=0$ and 
	$$g({\rm Ric}(T), \overline{T})=g(\mathfrak{R}_{|\mathfrak{g}}(T^\mathfrak{g}), \overline{T}^\mathfrak{g})$$
	for every $T\in\bigotimes^rT^*_\mathbb{C}M$.
\end{proposition}

Recall the Bochner formula for a tensor $ T \in \mathcal{T}^{(0,k)}(M)$ 
$$
\Delta \frac{1}{2}|T|^{2}=|\nabla T|^{2}-g\left(\nabla^{*} \nabla T, \overline{T}\right). 
$$
Therefore, if $T$ is harmonic, then $\nabla^*\nabla T=-c\Ric T$. Hence, this togeher with Proposition \ref{ric} implies
$$
\Delta \frac{1}{2}|T|^{2}=|\nabla T|^{2}+c \cdot g(\Ric({T}), \overline{T}).
$$ 
\begin{lemma}\label{lem2}
	Let $(V, g)$ be a Euclidean vector space, $\mathfrak{g}\subset\mathfrak{so}(V)$ a Lie subalgebra and let $\mathfrak{R}: \mathfrak{g}\to\mathfrak{g}$ be self-adjoint with eigenvalues $\mu_{1} \leq \ldots \leq \mu_{{\rm dim}\mathfrak{g}}$.  Let $T \in \bigotimes^rV^*_\mathbb{C}$. Suppose there is $C \geq 1$ such that
	$$
	|L T|^{2} \leq \frac{1}{C}|{T}^\mathfrak{g}|^{2}|L|^{2}
	$$
	for all $L \in \mathfrak{g}$. Let $1 \leq \ell\leq\lfloor C\rfloor$ be an integer and let $\kappa\leq0$. 
	\begin{enumerate}
	    \item If $\mu_1+\ldots+\mu_\ell+(C-\ell)\mu_{\ell+1}\geq\kappa(\ell+1)$, then $g(\mathfrak{R}(T^\mathfrak{g}), \overline{T}^\mathfrak{g}) \geq \frac{\kappa(\ell+1)}{C}|T^\mathfrak{g}|^2$,
	    \item if $\mu_1+\ldots+\mu_\ell+(C-\ell)\mu_{\ell+1}>0$, then $g(\mathfrak{R}(T^\mathfrak{g}), \overline{T}^\mathfrak{g}) >0$, unless ${T}^\mathfrak{g}=0$.
	    	\end{enumerate}
\end{lemma}
\begin{proof} See \cite[Lemma~1.8]{PW21}.
\end{proof}

Let $V=\mathbb{C}^n$ and consider the natural $U(n)$-action on $V$. Denote by 
$$\Lambda^{p,0}V^*=\Lambda^pV^*={\rm span}_{\mathbb{C}}\{dz^{i_1}\wedge\ldots\wedge dz^{i_p}| 1\leq i_1<\ldots< i_p\leq n\}$$
the space of complex linear $p$-forms, by
$$\Lambda^{0,q}V^*=\Lambda^q\overline{V^*}={\rm span}_{\mathbb{C}}\{d\bar{z}^{i_1}\wedge\ldots\wedge d\bar{z}^{i_q}| 1\leq i_1<\ldots< i_q\leq n\}$$
the space of conjugate linear $q$-form, and by
$$\Lambda^{p,q}V^*=\Lambda^{p,0}V^*\otimes_\mathbb{C}\Lambda^{0,q}V^*$$
the space of $(p,q)$-forms.

For $0\leq k\leq\min\{p,q\}$ set
$$V^{p,q}_k=\Lambda^{p-k,0}V^*\otimes_\mathbb{C}{\rm span}\{\Omega^k\}\otimes_\mathbb{C}\Lambda^{0, q-k}V^*,$$
where $\Omega$ is the K\"{a}hler form.
\begin{definition}\label{def25}
For $\varphi\in\Lambda^{p,q}V^*$, set
$$\overset{\circ}{\varphi}=\begin{cases}\varphi-\frac{g(\varphi, \Omega^p)}{|\Omega^p|}\Omega^p&\text{ if }p=q\\
\varphi&\text{ if }p\not=q
\end{cases}$$
\end{definition}
The following proposition calculate $\varphi^{\mathfrak{u}}$ in term of $\overset{\circ}{\varphi}$. 
\begin{proposition}\label{p32}
	Let $k\leq\min\{p, q\}$ and $\varphi\in\Lambda^{p,q}_kV^*$. We have
	$$|\varphi^{\mathfrak{u}}|=[2(p-k)(q-k)+(p+q-2k)((n+1)-(p+q-2k))]|\overset{\circ}{\varphi}|^2.$$
\end{proposition}
The next proposition allows us to estimate $|LT|^2$ for various types of tensors
\begin{proposition}\label{prop0}
    Suppose that $\varphi\in V^{p,q}_k$. It follows that
    $$|L\varphi|^2\leq(p+q-2k)|L|^2|\overset{\circ}{\varphi}|^2$$
    for all $L\in\mathfrak{u}(V)$.
\end{proposition}
	
\begin{proof}
	See \cite[Proposition~3.4]{PW21}. 
\end{proof}
For $k\leq\min\{p,q\}$, if $p+q-2k\not=0$, let $$C^{p,q}_k=n+1-(p+q)+2\frac{pq-k^2}{p+q-2k},$$
 and 
 $$C^{p,q}=n+1-\frac{p^2+q^2}{p+q}.$$
 Here, we can assume $p+q\leq n$ due to Serre duality.
\begin{proposition}\label{prop1}
	Let $k\leq\min\{p,q\}$ with $p+q-2k>0$. Let $\kappa\leq0$ and let $\mathfrak{R}:\mathfrak{u}(V)\to\mathfrak{u}(V)$ be a K\"{a}hler curvature operator with eigenvalues $\mu_1\leq\ldots\leq\mu_{n^2}$. Let $\varphi\in\Lambda^{p,q}_kV^*$.
	\begin{enumerate}
	    \item If $\mu_1+\ldots+\mu_{\lfloor C^{p,q}_k\rfloor}+(C^{p,q}_k-\lfloor C^{p,q}_k\rfloor)\mu_{\lfloor C^{p,q}_k\rfloor+1}\geq \kappa(\lfloor C^{p,q}_k)\rfloor+1)$ then
	    $$g(\mathfrak{R}(\varphi^\mathfrak{u}), \overline{\varphi}^\mathfrak{u})\geq \kappa (\lfloor C^{p,q}_k\rfloor+1)(p+q-2k)|\overset{\circ}{\varphi}|^2.$$
	    \item If $\mu_1+\ldots+\mu_{\lfloor C^{p,q}_k\rfloor}+(C^{p,q}_k-\lfloor C^{p,q}_k\rfloor)\mu_{\lfloor C^{p,q}_k\rfloor+1}>0$ then
	    $$g(\mathfrak{R}(\varphi^\mathfrak{u}), \overline{\varphi}^\mathfrak{u})>0$$
	    unless $\varphi=0$.
	    \item If $\mu_1+\ldots+\mu_{\lfloor C^{p,q}\rfloor}+(C^{p,q}-\lfloor C^{p,q}\rfloor)\mu_{\lfloor C^{p,q}\rfloor+1}\geq \kappa(\lfloor C^{p,q})\rfloor+1)$ then
	    $$g(\mathfrak{R}(\varphi^\mathfrak{u}), \overline{\varphi}^\mathfrak{u})\geq \kappa (n+2-|p-q|)(p+q)|\overset{\circ}{\varphi}|^2.$$
	\end{enumerate}
\end{proposition}
\begin{proof}
	See \cite[Proposition~3.6]{PW21} and \cite[Corollary~3.7]{PW21}.
\end{proof}
We note that, if $\mu_1+\ldots+\mu_{\lfloor C^{p,q}\rfloor}+(C^{p,q}-\lfloor C^{p,q}\rfloor)\mu_{\lfloor C^{p,q}\rfloor+1}\geq \kappa(\lfloor C^{p,q})\rfloor+1)$, then by \cite{PW21} (see the Proof of Theorem B-D), we have
$$g({\rm Ric}\varphi, \overline{\varphi})\geq \kappa (n+2-|p-q|)(p+q)|\overset{\circ}{\varphi}|^2.$$
Combining all above discussion, we have the following Bochner formula:
\begin{lemma}\label{lem1} Let $(M, g)$ be a complete non-compact K\"{a}hler manifold of complex dimension $n$, let $\kappa\leq0$. Suppose that $\varphi$ is a harmonic $(p,q)$-forms. If 
    $$\mu_1+\ldots+\mu_{\lfloor C^{p,q}\rfloor}+(C^{p,q}-\lfloor C^{p,q}\rfloor)\mu_{\lfloor C^{p,q}\rfloor+1}\geq \kappa(\lfloor C^{p,q})\rfloor+1),$$
    then we have
$$\Delta\frac{1}{2}|\varphi|^2\geq|\nabla\varphi|^2+\kappa (n+2-|p-q|)(p+q)|\overset{\circ}{\varphi}|^2.$$
\end{lemma}

\begin{lemma}\label{lem3}
Every algebraic K\"ahler curvature tensor $R\in Sym^2_{B}(\mathfrak{u}(n))$ satisfies
	    $$|\mathfrak{R}^\mathfrak{u}|^2= 4(n+1) |\overset{\circ}{ {R}}|^2-4|\overset{\circ}{{\rm Ric}}|^2.$$
In particular, $|\mathfrak{R}^\mathfrak{u}|^2=0$ if and only if $R$ has constant holomorphic sectional curvature. 
\begin{proof}
	\cite[Lemma 5.2]{PW}. 
\end{proof}
\end{lemma}

\subsection{Quaternion K\"ahler manifold}

A Riemannian manifold with holonomy contained in $Sp(m)\cdot Sp(1), m\geq 2$ is called quaternion K\"ahler manifold. Locally there exist almost complex structures $I,J,K$ such that $IJ=-JI=K$. For a local orthonormal frame $\left\{e_i,Ie_i,Je_i,Ke_i \right\}_{i=1,\cdots,m}$, consider
\begin{align*}
\omega_I&=\sum_{i=1}^{m}e_i \wedge Ie_i + Je_i \wedge Ke_i,\\
\omega_J&=\sum_{i=1}^{m}e_i \wedge Je_i + Ke_i \wedge Ie_i,\\
\omega_K&=\sum_{i=1}^{m}e_i \wedge Ke_i + Ie_i \wedge Je_i.
\end{align*}

It is straightforward to check that
$$g(IX,Y)=g(X\wedge Y,\omega_I),g(JX,Y)=g(X\wedge Y,\omega_J),g(KX,Y)=g(X\wedge Y,\omega_K).$$

The curvature operator of quaternionic projective space is given by
\begin{align*}
\mathfrak{R}_{\mathbb{HP}^m}(X\wedge Y)&=X\wedge Y+IX\wedge IY+JX\wedge JY+KX\wedge KY\\
&+2g(X\wedge Y,\omega_I)\omega_I+2g(X\wedge Y,\omega_J)\omega_J+2g(X\wedge Y,\omega_K)\omega_K.
\end{align*}

The curvature operator $R\in Sym^2_{B}(TM)$ of a quaternion K\"ahler manifold satisfies
\begin{equation*}
R=\frac{scal}{16m(m+2)}R_{\mathbb{HP}^m}+R_0,
\end{equation*}
where $R_0$ is the hyper-K\"ahler component. Hyper-K\"ahler manifolds are necessarily Ricci-flat. 

\begin{lemma}\label{lem4}
Let $m\geq 2$. An algebraic quaternion K\"ahler curvature tensor $R\in Sym^2_{B}(\mathfrak{sp}(m)\oplus \mathfrak{sp}(1))$ satisfies
	$$|\mathfrak{R}^{\mathfrak{sp}(m)\oplus \mathfrak{sp}(1)}|^2= \frac{4}{3}(3m+4)|R_0|^2 $$
	In particular, $\mathfrak{R}^{\mathfrak{sp}(m)\oplus \mathfrak{sp}(1)}=0$ if and only if $R$ has constant quaternionic sectional curvature. 
	\begin{proof}
		\cite[Corollary 4.5]{PW}. 
	\end{proof}
\end{lemma}

\subsection{K\"ahler manifolds with divergence free Bochner tensors}
Let $(M,J,g)$ be a K\"ahler manifold of real dimension $2m$. Let $\omega(X,Y)=g(JX,Y)$ denote the K\"ahler form and $\rho(X,Y)=\Ric(JX,Y)$ denote the Ricci form. The trace-free Ricci tensor is $\overset{\circ}{{\rm Ric}}=\Ric-\frac{\rm scal}{2m}g$ and the primitive part of the Ricci form is $\rho_0=\rho-\frac{\rm scal}{2m}\omega$.

The curvature tensor decomposes into a K\"ahler curvature tensor with constant holomorphic sectional curvature, a K\"ahler curvature tensor with trace-free Ricci curvature and the Bochner tensor,
\begin{align*}
\Rm =& \frac{\rm scal}{4m(m+1)}\left(\frac{1}{2}g \KN g+\frac{1}{2} \omega \KN \omega+2\omega \otimes \omega \right)\\
&+\frac{1}{2(m+2)}\left(\overset{\circ}{{\rm Ric}}\KN g + \rho_0 \KN \omega +2(\rho_0 \otimes \omega+\omega \otimes \rho_0) \right)+B.
\end{align*}

The Bochner tensor is totally trace-free \cite{Tachibana}. That is, if $e_1,\cdots,e_{2m}$ is an orthonormal basis of $TM$, then
\begin{equation*}
\sum_{i=1}^{2m}B(e_i,Y,e_i,W)=\sum_{i=1}^{2m}B(e_i,Je_i,Z,W)=0.
\end{equation*} 

\begin{proposition}\label{prop2}
Let $(M,g)$ be a K\"ahler manifold. If the Bochner tensor is divergence free, then it satisfies the second Bianchi identity and consequently 
\begin{equation*}
\nabla^{*} \nabla B+c \Ric(B)=0.
\end{equation*}
\begin{proof}
\cite[Proposition 3.2]{PW21}. 
\end{proof}
\end{proposition}

\section{Vanishing results}
	In this section, we assume that the curvature term $g(\mathfrak{R}_{|\mathfrak{g}}(T^\mathfrak{g}), \overline{T}^\mathfrak{g}) \geqslant -\kappa |T|^2$ for every harmonic tensor $T$ and for some $\kappa \geqslant 0$. We note that in the rest of this paper, we always assume that $Q\geq2$. For $\kappa =0$, we give a proof of Theorem \ref{thm1} as follows.

\begin{proof}[\textsc{Proof of Theorem \ref{thm1}}]
	For $T$ is a harmonic tensor, recall that the Bochner formula for harmonic tensor $T$ implies
	$$
	\Delta \frac{1}{2}|T|^{2}=|\nabla T|^{2}+c \cdot g(\mathfrak{R}_{|\mathfrak{g}}(T^\mathfrak{g}), \overline{T}^\mathfrak{g}).
	$$
	Since $g(\mathfrak{R}_{|\mathfrak{g}}(T^\mathfrak{g}), \overline{T}^\mathfrak{g}) \geqslant 0$,	same arguments from \cite{CDH} yield the proof. 
\end{proof}

	\begin{proof}[\textsc{Proof of Theorem \ref{thm2},\ref{thm2Q}}]
	Using the assumption $g(\mathfrak{R}_{|\mathfrak{g}}(T^\mathfrak{g}), \overline{T}^\mathfrak{g}) \geqslant -\kappa\rho|T|^2$, we obtain that 
$$c\kappa \int_{M}\rho \varphi^2 |T|^{q+2} \geqslant 2\int_{M} \varphi |T|^{q+1} \left< \nabla \varphi, \nabla |T| \right> + (q+1) \int_{M} \varphi^2|T|^q |\nabla|T||^2. $$

From same arguments from \cite{CDH}, for any $\kappa<\frac{4(Q-1)}{cQ^2}$ and $R>0$, it holds that
$$\int_{M} \varphi^2 |T|^{Q-2} |\nabla |T||^2 \leqslant \frac{4C}{R^2} \int_{M} |T|^{Q},$$
where $C=C(\varepsilon,q)>0$ and $\varepsilon$ is sufficient small real number. Let $R \to \infty$ and since $|T| \in L^Q(M)$, then $|T|$ is constant on each connected component of $M$. From the hypothesis~\eqref{eq:C2}, the volume of $M$ is infinite \cite[Corollary 3.2]{LPWJ}. By $|T| \in L^Q(M)$, it follows that $|T|=0$. Therefore, $T \equiv 0$. The proof is complete.  
\end{proof}
\begin{remark}\label{kato}
	If a refined Kato inequality 
	$$|\nabla T|^2\geq(1+a)|\nabla|T||^2$$
	holds true then as in \cite{CDH}, we can improve the upper bound of $\kappa$ to be
	$$\kappa<\frac{4(Q-1+a)}{cQ^2}.$$
\end{remark}
\begin{theorem}\label{thm3}
Let $(M, g)$ be a complete non-compact K\"{a}hler manifold of complex dimension $n$. If $p\not=q$ and
    $$\mu_1+\ldots+\mu_{\lfloor C^{p,q}\rfloor}+(C^{p,q}-\lfloor C^{p,q}\rfloor)\mu_{\lfloor C^{p,q}\rfloor+1}\geq 0,$$
    then every harmonic $(p,q)$-form is parallel. 
    
    In particular, if $\mu_1+\ldots+\mu_{\lfloor C^{p,q}\rfloor}+(C^{p,q}-\lfloor C^{p,q}\rfloor)\mu_{\lfloor C^{p,q}\rfloor+1}> 0$, then there are no nontrivial $(p,q)$-harmonic forms with finite $L^Q (Q\geq2)$-norm on $M$.
\end{theorem}
	\begin{proof}
	Let $\omega$ be a harmonic $(p,q)$-form with $|\omega| \in L^Q(M)$ for some $Q\geq2$. Applying Proposition \ref{lem1} and Proposition \ref{prop0}, we obtain that 
	\begin{equation}\label{eq6}
	|L\omega|^2 \leq (p+q-2k)|\overset{\circ}\omega|^2 |L|^2 = \frac{1}{C^{p,q}_k} |\omega^\mathfrak{u}|^2 |L|^2
	\end{equation} 
	for all $L \in \mathfrak{so}(TM)$.
	
	If $\mu_1+\ldots+\mu_{\lfloor C^{p,q}\rfloor}+(C^{p,q}-\lfloor C^{p,q}\rfloor)\mu_{\lfloor C^{p,q}\rfloor+1}\geq 0$, then the first conclusion of Lemma \ref{lem2} implies 
	$$
	g(\Re(\hat{\omega}), \hat{\omega}) \geq 0.
	$$
	An application of Theorem \ref{thm1} to Hodge Laplacian yields that $\omega$ is parallel. Moreover, $|\omega|$ is constant. Now, we can follow the argument as in Theorem \ref{thm1} to complete the proof. 
\end{proof}
We note that if $p=q$ then by Definition 2.5, we have $\overset{\circ}{\omega}=\omega$ if $\omega$ and $\Omega$ are perpendicular. Hence, using the proof of Theorem \ref{thm3}, we obtain the following result.
\begin{theorem}\label{thm3a}
	Let $(M, g)$ be a complete non-compact K\"{a}hler manifold of complex dimension $n$. If 
		$$\mu_1+\ldots+\mu_{\lfloor C^{p,p}\rfloor}+(C^{p,p}-\lfloor C^{p,p}\rfloor)\mu_{\lfloor C^{p,p}\rfloor+1}\geq 0,$$
	then every harmonic $(p,p)$-form $\omega$ is parallel provided that $\omega\perp\Omega$. 
	
	In particular, if $\mu_1+\ldots+\mu_{\lfloor C^{p,p}\rfloor}+(C^{p,p}-\lfloor C^{p,p}\rfloor)\mu_{\lfloor C^{p,p}\rfloor+1}> 0$, then there are no nontrivial $(p,p)$-harmonic forms $\omega$ with finite $L^Q (Q\geq2)$-norm on $M$ provided that $\omega\perp\Omega$.
\end{theorem}  

	Following \cite{Car}, the above result has a reduced $L^2$ cohomology interpretation as follows. Let $\mathcal{H}^{\ell}(M)$ be the space of $L^2$ harmonic $\ell$-forms, saysing $\mathcal{H}^\ell(M)=\{\omega\in L^2(\Lambda^{\ell} T^*M): d\omega=\delta\omega=0\}$, where $\delta$ is the dual of the differential operator $d$ and $Z^\ell_2(M)$ the kernel of the unbounded operator $d$ acting on $L^2(\Lambda^\ell T^*M)$, or equivalently
$$Z^\ell_2(M)=\{\omega\in L^2(\Lambda^\ell T^*M): d\omega=0\}.$$
The space $\mathcal{H}^\ell(M)$ can be used to characterize the \textit{reduced $L^2$ cohomology group} as follows 
$$\mathcal{H}^\ell(M)=Z^\ell_2(M)/\overline{dC_0^\infty(L^2(\Lambda^{\ell-1} T^*M)},$$
where the closure is taken with respect to the $L^2$ topology. It is worth to note that the finiteness of ${\rm dim}\mathcal{H}^\ell(M)$ depends only on the geometry of ends (\cite{Lott}). Observe that 
$$\Lambda^{\ell} T^{*}M=\underset{p+q=\ell}{\otimes}\Lambda^{p,q}T^{*}M.$$
Therefore, Theorem \ref{thm3} leads immediately to the following result.
\begin{corollary}
	Let $n \geq 3$ and let $(M, g)$ be a complete non-compact $n$-dimensional K\"ahler 	manifold. Then every harmonic $(p,q)$-form $\omega$ with $|\omega| \in L^Q(M)$ for some $Q\geq2$ is vanishing if the curvature tensor is $\lceil \frac{n}{2} \rceil$-nonnegative. In particular, every harmonic $(p,q)$-form $\omega$ with $|\omega| \in L^2(M)$ is vanishing, consequently, every reduced $L^2$ cohomology groups $\mathcal{H}^\ell(M)$ are trivial if $\ell$ is odd.
\end{corollary}

\begin{remark}\label{remd}
We recall that a $k$-form $\omega$ is said to be a harmonic field if $(d+d^*)\omega=0$, where $d$ is the differential operator and $d^*$ its dual operator.  If $M$ is compact then the harmonic
fields coincide with the harmonic forms. When $M$ is K\"{a}hler and $\omega$ is a $(p,q)$-form, it is easy to prove that $\omega$ is a harmonic field if and only if $\partial \omega = \partial^{*} \omega = \overline{\partial} \omega = \overline{\partial}^{*} \omega=0$.	We would like to mention that there is a refined Kato inequality for $(p,q)$-harmonic field $\omega$ (see \cite{DP12}), namely, there exists a constant $D^{p,q}\geq0$ such that 
	$$|\nabla \omega|^2\geq  \frac{1}{D^{p,q}}|\nabla|\omega||^2,$$	
		where 
	$$
	D^{p,q}= \begin{cases}\min\left\{\max\left\{\frac{2p+1}{2p+2},\frac{2n-2p+1}{2n-2p+2} \right\},\max\left\{\frac{2q+1}{2q+2},\frac{2n-2q+1}{2n-2q+2}\right\}     \right\}^2 & p,q\neq n \\ \frac{1}{{2}} & p=n \text{ or } q=n. \end{cases}
	$$
Hence, Remark \ref{kato} implies an improvement of the upper bound of $\kappa$ as follows
	$$0<\kappa< \frac{4(Q+\frac{1}{D^{p,q}}-3)}{cQ^2}.$$
Moreover, it is proved in Proposition 4.5 (see also Remark 3.12) in \cite{DP12} that if $\omega\in L^2(M)$ then $\omega$ is a harmonic field if and only if $\omega$ is harmonic. 
\end{remark}

The next results with a general curvature condition is a direct consequence of Theorem \ref{thm2} and the above Kato inequality.
\begin{theorem}\label{theo1}
	Let $(M,g)$ be a complete non-compact K\"ahler manifold of complex dimension $n$. Assume that  $M$ satisfies a weighted Poincar\'e inequality with a nonnegative weight function $\rho$ with ~\eqref{eq:C1} and ~\eqref{eq:C2}. Denote {$\mu_{1} \leq \ldots \leq \mu_{n^2} $} eigenvalues of the curvature operator of $(M,g)$. For $p\neq q$, if 
	$$
	\frac{\mu_1+\ldots+\mu_{\lfloor C^{p,q}\rfloor}+(C^{p,q}-\lfloor C^{p,q}\rfloor)\mu_{\lfloor C^{p,q}\rfloor+1}}{C^{p,q}+1} \geq -\kappa\rho
	$$ 
	then every $(p,q)$-harmonic field $\omega$ vanishes provided that $|\omega| \in L^Q(M)$ and 
	$$0< \kappa < \frac{4\left(Q+\frac{1}{D^{p,q}}-3\right)}{(n+2-|p-q|(p+q))Q^2} .$$
\end{theorem}
\begin{proof}
	Since $\omega$ is a $(p,q)$-harmonic field, using the Bochner formula in Lemma \ref{lem1}, we obtain that 
$$\Delta\frac{1}{2}|\varphi|^2\geq|\nabla\varphi|^2+\kappa (n+2-|p-q|)(p+q)\rho|\overset{\circ}{\varphi}|^2.$$
	Following the proof of Theorem \ref{thm2}, Remark \ref{kato} and Remark \ref{remd}, we complete the proof. 
\end{proof}
Observe that if $\omega\in L^2(M)$ and $\omega$ is harmonic then $\omega$ is a harmonic field. Therefore, Remark \ref{remd} and Theorem \ref{theo1} imply the following corollary.
\begin{corollary}\label{col1}
	Let $(M,g)$ be a complete non-compact K\"ahler manifold of complex dimension $n$. Assume that  $M$ satisfies a weighted Poincar\'e inequality with a nonnegative weight function $\rho$ with ~\eqref{eq:C1} and ~\eqref{eq:C2}. Denote {$\mu_{1} \leq \ldots \leq \mu_{n^2} $} eigenvalues of the curvature operator of $(M,g)$. For $p\neq q$, if 
	$$
	\frac{\mu_1+\ldots+\mu_{\lfloor C^{p,q}\rfloor}+(C^{p,q}-\lfloor C^{p,q}\rfloor)\mu_{\lfloor C^{p,q}\rfloor+1}}{C^{p,q}+1} \geq -\kappa\rho
	$$ 
	then every $L^2$ harmonic $(p,q)$-form $\omega$ vanishes provided that 
	$$0< \kappa < \frac{\frac{1}{D^{p,q}}-1}{n+2-|p-q|(p+q)} .$$
\end{corollary}
When $p=q$, we also can obtain a vanishing result as follows.
\begin{corollary}
	Let $(M,g)$ be a complete non-compact K\"ahler manifold of complex dimension $n$. Assume that  $M$ satisfies a weighted Poincar\'e inequality with a nonnegative weight function $\rho$ with ~\eqref{eq:C1} and ~\eqref{eq:C2}. Denote {$\mu_{1} \leq \ldots \leq \mu_{n^2} $} eigenvalues of the curvature operator of $(M,g)$. If 
	$$
	\frac{\mu_1+\ldots+\mu_{\lfloor C^{p,p}\rfloor}+(C^{p,p}-\lfloor C^{p,p}\rfloor)\mu_{\lfloor C^{p,p}\rfloor+1}}{C^{p,p}+1} \geq -\kappa\rho
	$$ 
	then every $L^2$ harmonic $(p,p)$-form $\omega$ vanishes provided that $\omega\perp\Omega$ and 
	$$0< \kappa < \frac{\frac{1}{D^{p,q}}-1}{n+2} .$$
\end{corollary}

Finally, the above corollary infers following vanishing result for reduced $L^2$ cohomology groups.
\begin{corollary}
	Let $(M,g)$ be a complete non-compact K\"ahler manifold of complex dimension $n$. Assume that  $M$ satisfies a weighted Poincar\'e inequality with a nonnegative weight function $\rho$ with ~\eqref{eq:C1} and ~\eqref{eq:C2}. Denote {$\mu_{1} \leq \ldots \leq \mu_{n^2} $} eigenvalues of the curvature operator of $(M,g)$. For $p\neq q$, if 
	$$
	\frac{\mu_1+\ldots+\mu_{\lfloor C^{p,q}\rfloor}+(C^{p,q}-\lfloor C^{p,q}\rfloor)\mu_{\lfloor C^{p,q}\rfloor+1}}{C^{p,q}+1} \geq -\kappa\rho
	$$ 
	then every harmonic $(p,q)$-form $\omega$, for all $1 \leq \ell \leq n-1$ vanishes provided that $|\omega| \in L^2(M)$ for some $\kappa$ satisfying
	$$ \kappa <  \frac{\frac{1}{D^{p,q}}-1}{n+2-|p-q|(p+q)}.$$
	Consequently, every reduced $L^2$-cohomology groups $\mathcal{H}^\ell(M)$ are trivial if $\ell$ is odd.
\end{corollary}

\section{Geometric applications}
\begin{theorem}\label{thm:app1}
Suppose that $(M,g)$ is a complete non-compact K\"{a}hler-Einstein manifold of complex dimension $n\geq4$. If $M$ satisfies a weighted Poincar\'e inequality and
$$\mu_1+\ldots+\mu_{\lfloor\frac{n+1}{2}\rfloor}+\frac{1+(-1)^n}{4}\mu_{\lfloor\frac{n+1}{2}\rfloor+1}\geq-k\rho$$
then $M$ is Riemannian flat provided that its $L^Q$-norm is finite, where $0\leq k<\frac{Q-1}{Q^2}$, $Q\geq2$.
\end{theorem}

	\begin{proof}
		
	As we mentioned in Example \ref{rmk1}.(b) that the curvature tensor of Einstein manifold is harmonic with respect to Lichnerowicz Laplacian $\Delta_{L} =\nabla^{*} \nabla + \frac{1}{2} \Ric$. 
	
	Since $M$ is Einstein, $\stackrel{\circ}{\Ric}=0$, so Lemma \ref{lem3} follows that  {$|{\mathrm{Rm}}^\mathfrak{u}|^{2}=4(n+1)|\stackrel{\circ}{\Rm}|^{2}$} and thus Lemma 2.2 in \cite{PW21} implies 
	\begin{equation}\label{eq7}
		|L \mathrm{Rm}|^{2} \leq8|\stackrel{\circ}{\Rm}|^{2}|L|^{2}=\frac{2}{n+1}{|{\mathrm{Rm}}^\mathfrak{u}|^{2}}|L|^{2}.
	\end{equation}
	for all $L \in \mathfrak{u}(n)$. By the assumption on the eigenvalues of the K\"ahler curvature operator and Lemma \ref{lem2} implies 
	$$g({\rm Ric}({\rm Rm}), \overline{\rm Rm})=
	{g(\mathfrak{R}_{|\mathfrak{u}}({\mathrm{Rm}^\mathfrak{u}}), \overline{\mathrm{Rm}^\mathfrak{u}}) \geq -\frac{2k\rho}{n+1}|\mathrm{Rm}^\mathfrak{u}|^2}\geq-8k\rho|\rm Rm|^2.
	$$
Here, we used ${|{\mathrm{Rm}}^\mathfrak{u}|^{2}}=4(n+1)|\stackrel{\circ}{\Rm}|^{2}\leq4(n+1)|{\rm  Rm}|^{2}$ since $M$ is Einstein. Applying Theorem \ref{thm2}, we obtain that $\Rm$ is vanishing. This means $M$ is flat.  
\end{proof}

	\begin{proof}[Proof of Theorem~\ref{d2}]
	
	Combining Lemma 2.2 in \cite{PW21} and Lemma 5.2 in \cite{PW21cre}, we have
	\begin{equation}
	| L B|^{2} \leq8|B|^{2}|L|^{2}=\frac{2}{n+1}|B^\mathfrak{u}|^{2}|L|^{2}.
	\end{equation}
	for all {$L \in \mathfrak{u}(n)$}. By the assumption on the eigenvalues of the K\"ahler curvature operator and Lemma \ref{lem2} implies 
	$$
		g(\Re(B^\mathfrak{u}), B^\mathfrak{u}) \geq -\frac{2k}{n+1}|B^\mathfrak{u}|^2\geq-8k\rho|B|^2.
	$$
	Applying Theorem \ref{thm2}, $B$ is vanishing. This means $M$ is Bochner flat, hence, using the clasification of Bochner flat geometry as in Theorem 5.7 in \cite{kam}, we complete the proof of the first conclusion. 
	
	Now, suppose that $M$ is either complex Euclidean geometry; or product of complex heperbolic and projective geometry. Then, $M$ has constant scalar curvature. By Theorem 1.1 in \cite{Kim09}, this together with the fact that $M$ has vanishing Bochner tensor and is irreducible implies $M$ is K\"{a}hler-Einstein. Hence, due to Theorem 1.2 in \cite{yanhui2012}, $M$ has constant holomorphic sectional curvature. The proof is complete.
\end{proof}
Recall that an $n$-dimensional K\"{a}hler manifold $(M^n, g)$ is called a gradient K\"{a}hler-Ricci soliton if there is a real-valued smooth function $f$ satisfying the soliton equation
$${R}_{i\bar{j}}+ \nabla_i\nabla_{\bar{j}}f = \lambda g_{i\bar{j}},$$
for some constant $\lambda\in\mathbb{R}$ and such that $\nabla f$ is a holomorphic vector field, i.e. $\nabla_i\nabla_jf=0$. Following the proof of Theorem \ref{d2}, we obtain the below proposition.
\begin{proof}[Proof of Proposition~\ref{d3}]
	Since the Bochner tensor has finite $L^Q$-norm, it must be vanishing. Hence the proof now follows by Theorem 1.2 in \cite{yanhui2012}.
\end{proof}
The final application is a rigidity result on quaternionic K\"{a}hler manifolds.
\begin{theorem}\label{quaternion}
Suppose that $(M,g)$ is a complete non-compact quaternion K\"{a}hler manifold of complex dimension $4m\geq8$. Let $\mu_1\leq\ldots\leq\mu_{m(2m+1)+3}$ denote the eigenvalues of the corresponding quaternion K\"{a}hler curvature operator. Suppose that the scalar curvature of $M$ is vanishing. If $M$ satisfies a weighted Poincar\'e inequality and
$$\mu_1+\ldots+\mu_{\lfloor\frac{m+1}{2}\rfloor}+\frac{5+3.(-1)^m}{12}{\mu}_{\lfloor\frac{m+1}{2}\rfloor+1}\geq-k\rho$$
then $M$ is Riemannian flat provided that the curvature tensor $R$ has finite $L^Q$-norm, where $0\leq k<\frac{Q-1}{Q}$, $Q\geq2$.
	\begin{proof}
	
	Quaternion K\"ahler manifolds in real dimension $4m\geq 8$ are Einstein. Hence the curvature tensor $R$ is harmonic and thus satisfies the Bochner formula 
$$\Delta \frac{1}{2}|R|^{2}=|\nabla R|^{2}+\frac{1}{2} \cdot g(\mathfrak{R}(R^{\mathfrak{sp}(m)\oplus \mathfrak{sp}(1)}), R^{\mathfrak{sp}(m)\oplus \mathfrak{sp}(1)}).$$
	
By Lemma \ref{lem4} and Lemma 2.2 in \cite{PW21}, we have 
\begin{equation}
| L R|^{2}=|LR_0|^2\leq8|L|^{2}|R_0|^{2}=\frac{6}{3m+4}|R^{\mathfrak{sp}(m)\oplus \mathfrak{sp}(1)}|^{2}|L|^{2}.
\end{equation}
for all $L \in \mathfrak{sp}(m)\oplus \mathfrak{sp}(1)$. By the assumption on the eigenvalues of the quaternion K\"ahler curvature operator and Lemma \ref{lem2} implies 
$$
g(\mathfrak{R}(R^{\mathfrak{sp}(m)\oplus \mathfrak{sp}(1)}), R^{\mathfrak{sp}(m)\oplus \mathfrak{sp}(1)})\geq-\frac{-6k\rho}{3m+4}|R^{\mathfrak{sp}(m)\oplus \mathfrak{sp}(1)}|^2\geq-8k\rho|R_0|^2=-8k\rho|R|^2.
$$
Here in the last inequality, we used the assumption that the scalar curvature is zero to infer $|R_0|=|R|$ (see the proof of Corollary 4.5 in \cite{PW}). Applying Theorem \ref{thm2}, $ R=0$. Hence, $M$ must be flat. 
\end{proof}
\begin{remark}
	If $k=0$, we always have 
	$$g(\mathfrak{R}(R^{\mathfrak{sp}(m)\oplus \mathfrak{sp}(1)}), R^{\mathfrak{sp}(m)\oplus \mathfrak{sp}(1)})\geq-\frac{-6k\rho}{3m+4}|R^{\mathfrak{sp}(m)\oplus \mathfrak{sp}(1)}|^2=0.$$ Hence, we can remove the condition on scalar curvature in the statement of Theorem \ref{quaternion}.
\end{remark}

\end{theorem}

\textbf{Acknowledgement}: The special thanks go to P. Petersen since his contribution is crucial to some results in this paper. He generously encouraged the authors to publish the results alone even though it should really be a joint paper. 
{The authors would like to express their deep thanks to the referee for his/her useful and critical comments  that lead to the improvement and correction of the presentation of this paper. The final version of this paper was completed during a visit of the second author at UCSB. The work of the second author was partially supported by a grant from the Niels Hendrik Abel Board under "Abel Visiting Scholar Program" of the International Mathematical Union (IMU). He would like to thank IMU for financial support, and Prof. Guofang Wei for her hospitality and encouragement.} We also greatly appreciate Prof. Ovidiu Munteanu's comments on the examples of weighted Poincar\'e inequality during the revision. 

\address{{\it Gunhee Cho}\\
	Department of Mathematics\\
	University of California, Santa Barbara\\
	552 University Rd, Isla Vista, CA 93117.}
{\\gunhee.cho@math.ucsb.edu}	
	\address{ {\it Nguyen Thac Dung}\\
	Faculty of Mathematics - Mechanics - Informatics \\
	Vietnam National University, University of Science, Hanoi \\
	Hanoi, Viet Nam and \\
	Thang Long Institute of Mathematics and Applied Sciences (TIMAS)\\
	Thang Long Univeristy\\ 
	Nghiem Xuan Yem, Hoang Mai\\
	Hanoi, Vietnam
}
{dungmath@vnu.edu.vn or {nguyenthac@ucsb.edu}}
%

\end{document}